\documentclass[11pt,reqno]{amsart}

\usepackage[utf8]{inputenc}
\usepackage{csquotes}
\usepackage[english]{babel}

\usepackage[toc, page]{appendix}

\usepackage[bookmarks=true]{hyperref}

\usepackage{amsmath}
\usepackage{amsthm}
\usepackage{amsfonts}
\usepackage{amssymb}
\usepackage{latexsym}
\usepackage{tikz}
\usetikzlibrary{calc, decorations.pathreplacing}
\usepgflibrary{decorations.markings}
\usepackage{subcaption}
\usepackage{dsfont}
\usepackage{array}
\usepackage{bm}
\usepackage{lmodern}
\usepackage[left=3cm,right=3cm,top=3cm,bottom=4cm]{geometry}

\usepackage{mdframed}
\usepackage{mathtools} 

\usepackage{enumitem}

\numberwithin{equation}{section}
\newtheorem{thm}{Theorem}[section]
\newtheorem{cor}[thm]{Corollary}
\newtheorem{lemme}[thm]{Lemma}
\newtheorem{prop}[thm]{Proposition}

\theoremstyle{definition}
\newtheorem{definition}[thm]{Definition}

\theoremstyle{remark}
\newtheorem{rmq}[thm]{Remark}

\def \Var {\textrm{Var}}

\newcommand{\N}{\mathbb{N}}
\newcommand{\R}{\mathbb{R}}
\newcommand{\1}{\mathds{1}}
\newcommand{\Proba}{\mathbb{P}} 
\newcommand{\E}{\mathbb{E}} 
\newcommand{\cvg}[3][ ]{ \underset{#2 \to #3}{\overset{#1}{\longrightarrow}}}

\DeclareMathOperator{\e}{e}

\newcommand{\mlaw}{\mathbf{m}}
\newcommand{\prcs}{Z}
\newcommand{\rnw}{M}

\newcommand{\espMes}{\mathcal{M}^1}

\begin{document}
\date{\today}

\title[Asymptotic deviation bounds for cumulative processes]{Asymptotic deviation bounds for cumulative processes}

\author[P. Cattiaux]{\textbf{\quad {Patrick} Cattiaux  \, }}
\address{{\bf {Patrick} CATTIAUX},\\ Institut de Math\'ematiques de Toulouse. CNRS UMR 5219. \\
	Universit\'e Paul Sabatier,
	\\ 118 route
	de Narbonne, F-31062 Toulouse cedex 09.} \email{patrick.cattiaux@math.univ-toulouse.fr}

\author[L. Colombani]{\textbf{\quad {Laetitia} Colombani  \, }}
\address{{\bf {Laetitia} COLOMBANI},\\ Institut de Math\'ematiques de Toulouse. CNRS UMR 5219. \\
	Universit\'e Paul Sabatier,
	\\ 118 route
	de Narbonne, F-31062 Toulouse cedex 09.} \email{laetitia.colombani@math.univ-toulouse.fr}

\author[M. Costa]{\textbf{\quad {Manon} Costa  \, }}
\address{{\bf {Manon} COSTA},\\ Institut de Math\'ematiques de Toulouse. CNRS UMR 5219. \\
	Universit\'e Paul Sabatier,
	\\ 118 route
	de Narbonne, F-31062 Toulouse cedex 09.} \email{manon.costa@math.univ-toulouse.fr}

\maketitle

\begin{center}
	\textsc{Universit\'e de Toulouse}
	\smallskip
\end{center}
\begin{abstract}
The aim of this paper is to get asymptotic deviation bounds via a Large Deviation Principle (LDP) for cumulative processes also known as compound renewal processes or renewal-reward processes. These processes cumulate independent random variables occurring in time interval given by a renewal process. Our result extends the one obtained in \cite{lefevere_large_2011} in the sense that we impose no specific dependency between the cumulated random variables and the renewal process and the proof uses \cite{mzsanov}. In the companion paper \cite{cattiaux_costa_colombani} we apply this principle to Hawkes processes with inhibition. Under some assumptions Hawkes processes are indeed cumulative processes, but they do not enter the framework of \cite{lefevere_large_2011}.
\end{abstract}
\bigskip


\bigskip

\bigskip

\section{Introduction.}
\subsection{Cumulative processes}
Cumulative processes have been introduced by Smith \cite{smith_1955} and are applied in many purposes, such as finance where they are called \emph{compound-renewal} processes or \emph{renewal-reward} processes. Indeed these continuous time processes cumulate independent random variables occurring in time interval given by a renewal process. To be more specific a real valued process $(Z_t)_{t\ge0}$ is called \emph{a cumulative process} if the following properties are satisfied: 
\begin{enumerate}
	\item $Z_0 = 0$, 
	\item there exists a renewal process $(S_i)_{i\ge0}$ such that for any $i$, $(Z_{S_i+t} - Z_{S_i})_{t \geq 0}$  is independent of $S_0, ... S_i$ and $(Z_s)_{s < S_i}$, 
	\item the distribution of $(Z_{S_i+t} - Z_{S_i})_{t \geq 0}$ is independent of $i$. 
\end{enumerate}
To study such processes, we write for all $t\geq 0$
$$Z_t = W_0(t) + W_1 + ... + W_{\rnw_t} + r_t,$$ 
where $W_0(t) = Z_{t \wedge S_0}$, $(W_i)_{i \geq 1}$ are i.i.d. random variables defined by $ W_i = Z_{S_i} - Z_{S_{i-1}},$ and $r_t$ is the remaining part $r_t = Z_t - Z_{\rnw_t}$ where $\rnw_t$ is the integer defined by $$\rnw_t = \sup \left\lbrace i\geq 0, S_i \leq t \right\rbrace. $$
We denote by $(\tau_i)_{i\ge1}$ the waiting times associated to the renewal process $\tau_i = S_i - S_{i-1}$. It is worth noticing that $\tau_i$ and $W_i$ can be dependent.

In the sequel we suppress the subscript $i$ when dealing with the distribution (and all associated quantities like expectation, variance ...) of $(\tau_i,W_i)$ and simply use $(\tau,W)$.
\medskip

A simple example of cumulative process is $Z_t = \int_{0}^{t} f(X_s) ds$ where $(X_t)_{t\ge0}$ is a regenerative process with i.i.d. cycles \cite{glynn_limit_1993}. Markov additive processes are other classical examples of cumulative process. 
In \cite{costa} the authors exhibited a renewal structure for some Hawkes processes. This description is extensively used in our companion paper \cite{cattiaux_costa_colombani} in order to describe such processes as cumulative processes, and to study their asymptotic behaviour. 
\medskip

For $\mathbb R$-valued cumulative processes, the law of large numbers (assuming that $\E[|W|]$ and $\E[\tau]$ are not infinite)
\begin{align*}
\frac{Z_t}{t} \cvg[a.s.]{t}{\infty} \frac{\E[W]}{\E[\tau]} \text{ if and only if } \E\left( \max_{S_0 \leq t < S_1} |r_t|  \right) < \infty \, ,
\end{align*}
and the central limit theorem (assuming $\Var (W)< \infty$ and $\Var (\tau) < \infty$)
\begin{align*}
\frac{\left(Z_t- t \frac{\E[W]}{\E[\tau]} \right)}{\sqrt{t}} \cvg{t}{\infty} \mathcal{N}\left(0, \sigma^2\right) \text{ where } \sigma^2= \frac{1}{\E(\tau)} \Var\left(W- \frac{\E[W]}{\E[\tau]} \tau\right)
\end{align*}
can be found in Asmussen \cite{asmussen}, theorem 3.1 and theorem 3.2.\\
Brown and Ross \cite{brown_asymptotic_1972} have proved an equivalent of Blackwell's  theorem and of the key renewal theorem for a subclass of cumulative processes, since cumulative processes are a generalization of renewal processes. Glynn and Whitt have focused in \cite{glynn_limit_1993} on cumulative processes associated to a regenerative process and have proved law of large numbers (strong and weak), law of the iterated logarithm, central limit theorem and functional generalizations of these properties.
\medskip

The aim of this work is to obtain asymptotic bounds in order to build confidence intervals. To this end we are looking at a large deviation principle (LDP) for cumulative processes. Some works have already been done. For instance, Duffy and Metcalfe \cite{duffy_how_2005} have considered the estimation of a rate function for a cumulative process (if it admits a LDP). 

In a series of papers, Borovkov and Mogulskii (\cite{borovkov_large_2015}, \cite{borovkov_large_2016_1}, \cite{borovkov_large_2016_2}) have studied the LDP (they use the term compound-renewal process), under some Cramer type assumptions. Actually, some points in their approach are not clear for us. After the submission of the present paper, Zamparo posted on ArXiv a preprint, now published in \cite{zamparo}, that extends Borovkov-Mogulskii approach, and is based on Cramer's theory. The same author had previously studied in \cite{zamparo2} the case of a discrete valued $\tau$.

Another possible approach based on a higher level LDP, namely at the level of empirical measures, was developed by Lefevere, Mariani and Zambotti \cite{lefevere_large_2011}. In this work they study specific cumulative processes where $W_i = F(\tau_i)$ for some deterministic function $F$ which is assumed to be non-negative, bounded and continuous. In a first version of this paper, we have extended their method to general pairs $(\tau,W)$ in $\mathbb R^+\times \mathbb R$. As suggested by the referee, our intricate proof can be simplified by using the Sanov type theorem obtained by Mariani and Zambotti in \cite{mzsanov}, what we shall do in the present work. Actually the proofs in \cite{mzsanov} greatly simplifies and extends the corresponding result for the empirical measure in\cite{lefevere_large_2011} (as well as our previous proof of this result).
\medskip

In this paper, we look at a LDP for $Z_t/t$ in the case $r_t = 0$ and $S_0=0$. This assumption can be relaxed if $r_t/t$ tends to $0$ quickly enough, as it will be the case for the application to Hawkes process (see \cite{cattiaux_costa_colombani}), we shall briefly recall. For example, if for all $\delta>0$
\begin{align*}
\limsup_{t \to \infty} \frac{1}{t} \log \Proba\left(\frac{|r_t|}{t} > \delta \right) = - \infty,
\end{align*}
then $Z_t/t$ and $(Z_t-r_t)/t$ are exponentially equivalent. They then admit the same asymptotic deviation bounds. 
\medskip

\subsection{Motivation: Application to Hawkes processes.}
A Hawkes process is a point process on the real line $\R$ characterized by its intensity process $t \mapsto \Lambda(t)$. We consider an appropriate filtered probability space $(\Omega, \mathcal{F}, (\mathcal{F}_t)_{t \geq 0}, \Proba)$ satisfying the usual assumptions. 
\begin{definition}
	\label{def:Hawkes}
Let $\lambda >0$ and $h: (0, +\infty) \rightarrow \R$ a signed measurable function. Let $N^0$ a locally finite point process on $(- \infty, 0]$ with law $\mlaw$. \\
	The point process $N^h$ on $\R$ is a Hawkes process on $(0, +\infty)$, with initial condition $N^0$ and reproduction measure $\mu(dt) = h(t) dt$ if: 
	\begin{itemize}
		\item $N^h\mid_{(-\infty, 0]} = N^0$,
		\item the conditional intensity measure of $N^h\mid_{(0, +\infty)}$ with respect to $(\mathcal{F}_t)_{t \geq 0}$ is absolutely continuous w.r.t the Lebesgue measure and has density:
		\begin{equation} \label{eq_intensite}
		\Lambda^h : t \in (0, +\infty) \mapsto f\left( \lambda + \int_{(-\infty, t)} h(t-u) N^h(du) \right) \, .
		\end{equation}
	\end{itemize}
	for some non-negative function $f$.
\end{definition}

Hawkes processes have been introduced by Hawkes \cite{hawkes71}. Most of the literature concerned with the large time behaviour of $N^h_t = N^h([0,t])$ is dedicated to the case $h\geq 0$ (self excitation).  This behaviour is studied in details in \cite{cattiaux_costa_colombani} when $h$ is a signed (the negative part modelling self inhibition) compactly supported function, and the function $f$ (called the jump rate function) is given by $$f(u) = \max(0,u) \, . $$ 
In this situation one gets a description of $N_t^h$ as a cumulative process (see \cite{cattiaux_costa_colombani} subsection 2.3) with few information on the joint law of $(\tau,W)$. This was the initial motivation for the present work. In particular, controlling the asymptotic deviation from the mean, in this framework with unbounded $W_i$'s, can lead to asymptotic confidence intervals. We refer to Corollary 2.13  \cite{cattiaux_costa_colombani} for a more complete overview and explicit results in this situation. We shall discuss this situation later.
\medskip

\section{Notations and main result}\label{sec_LDP_notandaim}
\subsection{First notations. \\ \\} \label{subsec_LDP_firstnotations}
We consider $(\tau_i, W_i)_{i\ge1}$ an i.i.d. sequence of pairs of random variables built on some probability space $(\Omega,\mathcal F,\mathbb P)$ with values in $[0,+\infty]\times \mathbb R$. Actually we are mainly interested in the case where $W$ takes non-negative values which is the case for Hawkes processes. 

The law of $(\tau_i, W_i)$ is an arbitrary probability measure $\psi$ on $(0, +\infty) \times \mathbb R$. We denote this by: $(\tau_i, W_i) \sim \psi$. In the sequel we generically use the notation $(\tau,W)$ for a pair with the same distribution as $(\tau_i,W_i)$. Notice that we thus assume that $$\psi(\tau=0)=\psi(\tau=+\infty)=0$$ which is Assumption (A1) in \cite{mzsanov}, implying in in particular that $\mathbb E(\tau)>0$.

We denote by $\espMes(\mathcal{X})$ the space of probability measure on some measurable space $(\mathcal{X},\mathcal G)$.
\medskip

We consider the renewal process associated with $(\tau_i)_{i\geq 1}$ : 
\begin{align*}
S_0&=0, \hspace{5 pt } S_n = \sum_{i=1}^{n} \tau_i,\\
\rnw_t &= \sup \left\lbrace n\geq 0, S_n \leq t \right\rbrace.
\end{align*}
We will study the quantity: 
\begin{equation}\label{eqz}
\prcs_t = \sum_{i=1}^{\rnw_t} W_i,
\end{equation}
where as usual an empty sum is equal to $0$.

The first main goal of this paper is to prove a Large Deviation Principle for the process $(Z_t/t)_{t\ge0}$. Let us recall some basic definitions in large deviation theory (we refer to \cite{dembo_large_2010}).

A family of probability measures $(\eta_t)_{t\ge0}$ on a topological space $(\mathcal X,T_{\mathcal X})$ equipped with its Borel $\sigma$-field, satisfies the Large Deviations Principle (LDP) with rate function $J(.)$ and speed $\gamma(t)=t$ if $J$ is lower semi-continuous from $\mathcal X$ to $[0,+\infty]$, and the following holds
\begin{equation}\label{eqLDP1}
- \inf_{x \in \mathcal O} \, J(x) \, \leq \, \liminf_{t \to +\infty} \, \frac{1}{t} \, \log \eta_t(\mathcal O) \quad \textrm{for all open subset $\mathcal O$,}
\end{equation}
and
\begin{equation}\label{eqLDP2}
- \inf_{x \in \mathcal C} \, J(x) \, \geq \, \limsup_{t \to +\infty} \, \frac{1}{t} \, \log \eta_t(\mathcal C) \quad \textrm{for all closed subset $\mathcal C$.}
\end{equation}
We shall say that $(\eta_t)_{t\ge0}$ satisfies the \textit{full} LDP when \eqref{eqLDP1} and \eqref{eqLDP2} are satisfied, while we will use \textit{weak} LDP when $\mathcal C$ closed is replaced by $\mathcal C$ compact in \eqref{eqLDP2}. When $\eta_t$ is the distribution of some random variable $Y_t$ (for instance $Z_t/t$) we shall say that the family $(Y_t)_{t\ge0}$ satisfies a LDP.

Since $J$ is lower semi-continuous the level sets $\{x\in\mathcal X, J(x)\leq a\}$ are closed. If in addition they are compact, then $J$ is said to be a good rate function.

In this paper we only consider the speed function $\gamma(t)=t$ so that we will no more refer to it.
\medskip

A particularly important notion for our purpose is the notion of \textit{exponentially good approximation}. 
\begin{definition}\label{defexpgood}
	Assume that $(\mathcal X,d)$ is a metric space. A family of random variables $\{(Y_{n,t})_{t\ge0}\}_{ n\in\N}$ is an exponentially good approximation of $(Y_t)_{t\ge0}$ (all these variables being defined on the same probability space $(\Omega,\mathbb P)$), if for all $\delta>0$ it holds $$\lim_{n \to \infty} \, \limsup_{t \to \infty} \, \frac 1t \, \log \mathbb P(d(Y_{n,t},Y_t)>\delta) \, = \, -\infty \, .$$
\end{definition}
The key result is then
\begin{thm}\label{thmexpgood}
	In the framework of definition \ref{defexpgood}, assume that $\{(Y_{n,t})_{t\ge0}\}_{ n\in\N}$ is an exponentially good approximation of $(Y_t)_{t\ge0}$. Then the following statements hold true.
	\begin{enumerate}
		\item If $\{(Y_{n,t})_{t\ge0}\}_{ n\in\N}$ satisfies a full LDP with rate function $J^n$ then $(Y_t)_{t\ge0}$ satisfies a weak LDP with rate function 
$$J(x) = \sup_{\delta>0} \, \liminf _{n \to \infty} \, \inf_{d(y,x)<\delta} \, J^n(y) \, .$$
		\item If $\mathcal X$ is locally compact, then the same conclusion is true when $\{(Y_{n,t})_{t\ge0}\}_{ n\in\N}$ satisfies only a weak LDP.
		\item If $J$ (defined above) is a good rate function such that for any closed set $F$, 
		$$\inf_{y \in F} J(y) \leq \limsup_{n \to \infty} \, \inf_{y \in F} J^n(y)\,,$$
		 then $(Y_t)_{t\ge0}$ satisfies a full LDP with rate function $J$.
	\end{enumerate}
\end{thm}
The first and last points in the previous Theorem are contained in \cite{dembo_large_2010} Theorem 4.2.16. The second one is a consequence of the fact that closed balls are compact sets.
Usually, the Theorem is sufficient to prove a full LDP. Nevertheless, it some cases, the study of the rate function $J$ is difficult. The lemma below gives an alternative, using exponential tightness which is easy to obtain with our assumptions.

\begin{lemme}\label{lem_exptightLDP}
	If $(Y_t)_{t\ge0}$ satisfies a weak LDP with a rate function $I$ and is exponentially tight, i.e. for all $\alpha>0$, there exists a compact set $K_{\alpha}$ such that
	$$\limsup_{t \to \infty} \frac{1}{t} \log \Proba \left( Y_t \notin K_{\alpha}^c  \right) < - \alpha,$$
	then $(Y_t)_{t\ge0}$ satisfies a full LDP and $I$ is a good rate function. 
\end{lemme}
This Lemma is a consequence of the Lemma 1.2.18 in \cite{dembo_large_2010}. 
\medskip

\subsection{Main results. }\label{subsecmainr}

Introduce the following quantities
\begin{equation}\label{eqbornes1}
\theta_0 := \sup_{\theta \geq 0 }\, \{ \,\E[\e^{\theta \tau}] < \infty\}  \, , \end{equation}
and
\begin{equation}\label{eqbornes2}
\eta_0 := \sup_{\eta\ge0}\, \{\, \E[\e^{\eta |W|}] < \infty\}  \, .
\end{equation}

Also introduce the classical Cramer transform, for $(a,b) \in \mathbb R^2$,
\begin{equation}\label{eqcramer1}
\Lambda^*(a,b) = \sup_{(x,y) \in\mathbb R^2} \, \{ax+by - \log \mathbb E(e^{x \tau +y W})\} \, . 
\end{equation}
We finally define, for $(m,\beta,x,y)\in \mathbb R^4$,
\begin{equation}\label{eqcramer2}
\Lambda(m,\beta,x,y)=x +my - \beta \, \log\mathbb E\left(e^{x \tau +y W}\right)
\end{equation}
 and the rate function $J$ for any $m\in \mathbb R$,
\begin{eqnarray}\label{eqrate}
J(m) &=& \inf_{\beta>0} \; \beta \, \Lambda^*\left(\frac 1\beta, \frac m\beta\right) \, , \nonumber \\ &=& \inf_{\beta>0} \, \sup_{x,y}\Lambda(m,\beta,x,y) \, .
\end{eqnarray}

We then may state
\begin{thm}\label{thmmain}
Assume that $\eta_0>0$ and $\theta_0>0$. Let $J$ given by \eqref{eqrate} and $\bar J$ defined as 
$$\begin{aligned}
&\bar J(m)=J(m) \text{ for } m\neq 0,\\
& \bar J(0) = \min(J(0),\theta_0)\,. 
\end{aligned}$$
\begin{itemize}
\item If $\eta_0=+\infty$ (in particular if $W$ is bounded) then $(Z_t/t)_{t\ge0}$ satisfies a full LDP with good rate function $\bar{J}$. 
\item If $\eta_0 < +\infty$, denoting $m=\mathbb E(W)/\mathbb E(\tau)$ we have for all $a>0$ and $\kappa \in (0,1)$
\begin{equation}\label{eq_ineqgeneral_J1}
\limsup_{t \to \infty} \, \frac 1t \, \log \mathbb P\left(\frac{Z_t}{t} \geq m +a\right) \leq - \, \min\left[\inf_{z\geq m+\kappa a} \bar J(z) \; , \;  \frac{\eta_0 a(1-\kappa)}{4}  \right] \, ,
\end{equation}
and similarly
\begin{equation}\label{eq_ineqgeneral_J2}
\limsup_{t \to \infty} \, \frac 1t \, \log \mathbb P\left(\frac{Z_t}{t} \leq m -a\right) \leq - \, \min\left[\inf_{z\leq m-\kappa a} \bar J(z) \; , \; \frac{\eta_0 a(1-\kappa)}{4}    \right] \, .
\end{equation} 
\end{itemize}
\end{thm}
\medskip

\begin{rmq} 
\textbf{A short discussion.}
As we said in \cite{cattiaux_costa_colombani}, the direct Cramer's approach in e.g. \cite{zamparo} furnishes more general results but with a much less explicit rate function. 

In particular, contrary to \cite{zamparo}, when $\eta_0<+\infty$ we do not provide a LDP principle but asymptotic deviation bounds. These bounds are actually what is useful from a statistical point of view, since they allow to build confidence intervals around the asymptotic mean. 

Due to the fact that we are using the results in \cite{mzsanov}, the method we will develop here extends immediately to $W$ taking its values in $\mathbb R^k$ or even in a general infinite dimensional normed vector space, provided $\theta_0=+\infty$ in the latter case. Actually, most of the work in the present paper is about understanding the rate function, and giving a tractable form for it.
\hfill $\diamondsuit$
\end{rmq}
\medskip

\section{Large Deviations for the empirical measure.}\label{secempir}

Following \cite{mzsanov}, we introduce the empirical measure 
\begin{equation}\label{eqempir}
\mu_t := \, \frac 1t \, \int_{[0,t)} \, \delta_{(\tau_{M_s+1},W_{M_s+1})} \, ds\,,
\end{equation}
so that, considering $$\varphi(u,w)=\frac{w}{u}$$ one has
\begin{align}
\label{eqempircontrac}
\mu_t(\varphi)&:= \int\, \varphi \, d\mu_t 
= \frac{1}{t} \int_{0}^t \frac{W_{M_s+1}}{\tau_{M_s+1}} ds \nonumber\\
&= \frac{1}{t}  \sum_{i=1}^{M_t} \int_{S_{i-1}}^{S_i} \frac{W_{i}}{\tau_{i} }  ds +\frac{1}{t}\int_{S_{M_t}}^t \frac{W_{M_t+1}}{\tau_{M_t+1}} ds \nonumber\\
& = \, \frac {Z_t}{t} \, + \, \frac{t-S_{M_t}}{t} \, \frac{W_{M_t+1}}{\tau_{M_t+1}} \, ,
\end{align}
if the latter makes sense.

We will thus deduce a LDP for $(Z_t/t)_{t\ge0}$ from a LDP for $(\mu_t)_{t\ge0}$ and the contraction principle (\cite{dembo_large_2010} Theorem 4.2.1). The LDP for $(\mu_t)_{t\ge0}$ is precisely the aim of the work by Mariani and Zambotti \cite{mzsanov}. We have to introduce some more notations. 
\medskip

First, for the sake of simplicity we still assume that $\mathcal X=(0,+\infty)\times \mathbb R$ so that Assumption (A4) (i.e. $\mathcal X$ locally compact) in \cite{mzsanov} is satisfied. The generic point in $\mathcal X$ is denoted by $x=(u,w)$. The application denoted by $\tau$ in \cite{mzsanov} is thus simply $(u,w) \mapsto u$ in our setting.

This immediately implies that Assumption (A2) in \cite{mzsanov} is satisfied, since for all $x=(u,w)\in (0,+\infty)\times \mathbb R$ it holds $$\zeta(x)=\inf_{\delta>0} \, \sup\left\{c\geq 0 \, : \, \int_{B((u,w),\delta)} \, e^{cu'} \, \psi(du',dw') < +\infty\right\} \, = + \infty \, .$$ Assumption (A3) therein is equivalent to $\theta_0=+\infty$ and we shall not use it.
\medskip

The set of non-negative Radon measures on $\mathcal X$ with total mass less than or equal to $1$ is denoted by ${\bar{\mathcal M}}^1(\mathcal X)$. The main advantage of considering this set is that it is compact and Polish for the vague topology i.e. the weakest topology such that for any continuous and compactly supported $f$, the map $\nu \mapsto \int f \, d\nu:= \nu(f)$ is continuous. Recall that if $f$ is continuous, bounded and goes to $0$ at infinity (i.e. $\sup_{|x|>R} |f(x)| \to 0$ as $R \to \infty$), then the application $\nu \mapsto \nu(f)$ is continuous on ${\bar{\mathcal M}}^1(\mathcal X)$.\\
We denote by $\mathcal{M}^1(\mathcal X)$ the set of probability measures on $\mathcal X$.
In \cite{lefevere_large_2011} to $\nu \in {\bar{\mathcal M}}^1(\mathcal X)$ is associated the probability measure $$\widetilde{\nu}(dx)=\nu(dx) + (1-\nu(\mathcal X))\delta_\partial$$ where $\mathcal X \cup \partial$ denotes the one point compactification of $\mathcal X$. 

In both papers the authors then introduce, provided $0 < \nu(1/u):=\int \frac{1}{u} \, \nu(du,dw) < +\infty$,
\begin{equation}\label{eqpondere}
\bar \nu(dx)= \bar \nu(du,dw) := \frac{1}{\nu(1/u)} \, \frac{1}{u} \, \nu(du,dw) \, .
\end{equation}

Finally recall that if $\pi$ and $\pi'$ are probability measures on $\mathcal X$, the relative entropy of $\pi$ w.r.t. $\pi'$ is defined as
\begin{align*}
H(\pi|\pi') = \begin{dcases}
\int \log\left( \frac{d\pi}{d\pi'} \right) d\pi &\text{ if } \pi \text{ is absolutely continuous w.r.t. } \pi' \\
+ \infty &\text{ otherwise.}
\end{dcases}
\end{align*}

Since Assumptions (A1), (A2) and (A4) are satisfied, Proposition 1.5 and Theorem 1.6 in \cite{mzsanov} then imply in our framework
\begin{thm}\label{thmmz}
Define $I: {\bar{\mathcal M}}^1(\mathcal X) \to [0,+\infty]$ as 
\begin{equation}\label{eqI1}
I(\nu) = \left\{ \begin{aligned}
&\nu(1/u) H(\bar{\nu} | \psi) + (1-\nu(\mathcal X))\theta_0 \, ,\qquad \text{if }\,0<\nu(1/u)<+\infty\\
&\theta_0\,, \quad \text{if }\nu\text{ is the null measure}\\
&+\infty,\, \text{otherwise.}
\end{aligned}\right.
\end{equation}
Then $I$ is convex, is a good rate function and the family $(P_t)_{t\geq 0}$ of the probability distributions of $(\mu_t)_{t\ge0}$ satisfies a full LDP with rate function $I$ and speed $t$.
\end{thm}
The specific case where $\nu$ is the null measure will play a special role. Notice that under our hypotheses the null measure is the only one such that $\nu(1/u)=0$.
\medskip

An immediate corollary can then be obtained using the contraction principle in a specific case.
\begin{cor}\label{corcontrac}
Assume in addition that there exists positive constants $K$ and $\varepsilon<1$ such that $\psi(|w|\leq K \textrm{ and } u \geq \varepsilon)=1$. Then, $(\mu_t(\varphi))_{t\ge0}$ satisfies a full LDP with the convex good rate function 
\begin{equation}\label{eqbarJ1}
\bar J(m)=\inf \, \{I(\nu) \, , \, \nu \in {\bar{\mathcal M}}^1(\mathcal X): \, \nu(\varphi)=m\} \, ,
\end{equation}
where as usual the infimum on an empty set is $+\infty$.
\end{cor}
\begin{proof}
Let $\eta_K$ be a continuous function such that $\mathbf 1_{|w|\leq K} \leq \eta_K(w) \leq \mathbf 1_{|w| \leq 2K}$. Introduce $$\varphi_{K,\varepsilon}(u,w)= \frac{w}{u\vee \varepsilon} \; \eta_K(w) \,  .$$
First remark that under our assumptions on $\psi$, $\mu_t(\varphi)=\mu_t(\varphi_{K,\varepsilon})$ almost surely. Since $\varphi_{K,\varepsilon}$ is continuous, bounded and goes to $0$ at infinity, $\nu \mapsto \nu(\varphi_{K,\varepsilon})$ is continuous. One can thus apply the contraction principle, yielding a full LDP with good rate function $$\bar J_{K,\varepsilon}(m)=\inf \, \{I(\nu) \, , \, \nu \in {\bar{\mathcal M}}^1(\mathcal X): \, \nu(\varphi_{K,\varepsilon})=m\} \, .$$ If one of $\bar J$ or $\bar J_{K,\varepsilon}$ is finite then $\nu$ is necessarily absolutely continuous w.r.t. $\psi$ (including the case of the null measure) so that $|W|\leq K$ and $\tau\geq \varepsilon$, $\nu$ almost everywhere.  Accordingly $\nu(\varphi)=\nu(\varphi_{K,\varepsilon})$ and $\bar J=\bar J_{K,\varepsilon}$.
\end{proof}

To obtain our main result, it remains to relax the boundedness assumptions on $\tau$ and $W$ and to compare $\bar J$ and $J$ defined in \eqref{eqbarJ1} and \eqref{eqrate}. The next result is a first step in this direction, removing the assumption on $\tau$.
\begin{prop}\label{propJegal}
Assume that there  exists a positive constant $K$ such that $\psi(|w|\leq K)=1$. Then
for $m \neq 0$, $\bar J(m)=J(m)$ while for $m=0$, $\bar J(0)=\min(J(0),\theta_0)$, where $J$ is defined in \eqref{eqrate}. 
\end{prop}
\begin{proof}
The proof is inspired by the proof of Lemma 5.1 in \cite{lefevere_large_2011}.

First remark that if $\nu \in {\bar{\mathcal M}}^1(\mathcal X)$, introducing the normalized $\nu_1=\nu/\nu(\mathcal X)$ (except if $\nu=0$), one has on the one hand $\bar{\nu_1}=\bar \nu$ and on the other hand 
$$I(\nu)=\nu(\mathcal X) \, \nu_1(1/u) \, H(\bar{\nu_1}|\psi) + (1-\nu(\mathcal X)) \, \theta_0 \, ,$$ provided $\nu(1/u)<+\infty$. 

Since for a non null $\nu$, $\nu(\mathcal X)$ can be any $\alpha \in ]0,1]$, we deduce that, defining 
\begin{equation*}
\begin{aligned}
\bar J_1(m) = \inf \left\{\alpha \right. &\, \nu_1(1/u) \, H(\bar{\nu}_1|\psi) + (1-\alpha) \theta_0 \, ; 
\\ 
\,& \left.\alpha \in ]0,1],\nu_1 \in \mathcal M^1(\mathcal X),\nu_1(1/u)<+\infty ,\nu_1(\varphi)=\frac{m}{\alpha} \, \right\} \, ,
\end{aligned}
\end{equation*}
one has $$\bar J(m)=\bar J_1(m) \textrm{ for $m\neq 0$} \quad ; \quad \bar J(0)= \min(\bar J_1(0),\theta_0) \, ,$$ since for $m=0$ one has to also consider the null measure. 

Since $\bar \nu(w)=  \nu(\varphi)/\nu(1/u)$ and $\nu(1/u) = 1/\bar \nu(u)$, it is elementary to see that 
\begin{equation}
\label{eqbarJbis}
\bar J_1(m) = \inf_{\alpha \in ]0,1]\,,\gamma >0\,,\nu' \in \mathcal M^1(\mathcal X)} \{(\alpha/\gamma) \, H(\nu'|\psi) + (1-\alpha) \theta_0; \nu'(u)=\gamma,\nu'(w) = \gamma \, m/\alpha \} \, ,
\end{equation}
the correspondence being $\nu'=\bar \nu_1$ i.e $\nu_1 = (1/\nu'(u)) \, u \, \nu'$.

Now we can mimic what is done in \cite{lefevere_large_2011}. 
 
Let $p(a,b) = \inf \{ H(\nu' |\psi)\,; \nu' \in \mathcal M^1(\mathcal X),\nu'(u) = a, \nu'(w) = b  \}$. We have 
\begin{eqnarray*}	
p^*(x,y) &=& \sup_{a,b\in \mathbb R^2}(ax+by-p(a,b))\\
&=&\sup_{a,b \in \mathbb R^2, \nu' \in \mathcal M^1(\mathcal X)}\{ax+by-H(\nu' |\psi); \nu'(u)=a, \nu'(w)=b\}\\
&=& \sup_{\nu'\in \mathcal M^1(\mathcal X)}\{\nu'(xu+yw)-H(\nu' |\psi)\} = \log \psi(e^{x\tau+yW})\\& = &\Lambda(x,y)
\end{eqnarray*}
thanks to the variational definition of the relative entropy. Since $p$ is lower semi continuous and convex we have $p=(p^*)^*=\Lambda^*$.

We thus deduce that
$$\bar{J_1}(m) = \inf \left\{ \frac{\alpha}{\gamma} \Lambda^*\left(\gamma, \frac{m \gamma}{\alpha} \right) + (1-\alpha) \theta_0 \,;  \alpha \in ]0,1], \gamma >0   \right\} \, .$$	But
	\begin{align*}
	\frac{\alpha}{\gamma} \Lambda^*\left(\gamma, \frac{m \gamma}{\alpha} \right) 
	&=  \beta \Lambda^*\left( \frac{\alpha}{\beta}, \frac{m}{\beta} \right)\text{ where } \beta = \frac{\alpha}{\gamma}.
	\end{align*}
	Thus
	\begin{align*}
	\bar{J_1}(m)&= \inf \left\{ \beta \Lambda^*\left( \frac{\alpha}{\beta}, \frac{m}{\beta} \right) + (1-\alpha) \theta_0\, ;\alpha \in ]0,1], \beta >0   \right\}. 
	\end{align*}
	
We will show that, for any $\beta>0$
	\begin{align*}
	\inf_{\alpha \in ]0,1]} \left\{\beta \Lambda^*\left( \frac{\alpha}{\beta}, \frac{m}{\beta} \right) + (1-\alpha) \theta_0 \right\} = \beta \Lambda^*\left(\frac{1}{\beta}, \frac{m}{\beta}\right).
	\end{align*}
Taking $\alpha = 1$, we see that the left hand side is less than or equal to the right hand side. To show the converse inequality, let us pick $\alpha \in ]0,1]$: 
	\begin{align*}
	\beta \Lambda^*\left(\frac{\alpha}{\beta}, \frac{m}{\beta}\right) + (1-\alpha) \theta_0 &= \sup_{x,y\in\mathbb R^2} \{ \alpha x + (1-\alpha) \theta_0  + m y - \beta \Lambda(x,y) \}\\
	&\geq \sup_{x,y\in\mathbb R^2} \{ x \wedge \theta_0 + my - \beta \Lambda(x,y) \}.
	\end{align*}
	Since $W$ is bounded, $e^{yW} \geq C(y) >0$ for all $y$, so that we have for all $x > \theta_0$  and all $y$, 
$$ \psi(\e^{x\tau +y W}) \geq C(y) \, \psi( \e^{x \tau}) = + \infty \, .$$ This shows  that $\Lambda(x,y) = +\infty$, for all $x > \theta_0$ and for all $y$. Hence, the supremum on $x$ can be restricted to the supremum on $\{x \leq \theta_0\}$: 
	\begin{align*}
	\beta \Lambda^*\left(\frac{\alpha}{\beta}, \frac{m}{\beta}\right) + (1-\alpha) \theta_0 &\geq \sup_{x,y\in\mathbb R^2} \{ x \wedge \theta_0 + my - \beta \Lambda(x,y) \}\\
	&= \sup_{x \leq \theta_0,y\in \mathbb R} \{ x + my - \beta \Lambda(x,y) \}\\
	&= \beta \Lambda^*\left(\frac{1}{\beta}, \frac{m}{\beta} \right)
	\end{align*}
and the desired inequality is proved.
\end{proof}
\medskip

\begin{rmq}\label{remlsc}
Let us remark on a simple example that the rate function $J$ defined in \eqref{eqrate} is not lower semi continuous.
If $W=1$, one has $Z_t=M_t$ and one easily sees that (recall \eqref{eqcramer2}) $\sup_{x,y\in\mathbb R^2}\Lambda(m,\beta,x,y)=+\infty$ except for $\beta=m$ yielding $J(m)=\sup_x \{x-m\log \mathbb E\left(e^{x\tau}\right)\}$ as expected. Notice that $J(0)=+\infty$ since $\beta>0$. In particular if $\tau$ has an exponential distribution with parameter $1$, $\theta_0=1$, $Z_t$ is the standard Poisson process and $J(m)=1-m+m\log m$ for $m>0$ while $J(m)=+\infty$ if $m\leq 0$. Accordingly $J$ is not lower semi continuous at $m=0$, and $\bar J$ is precisely the lower semi continuous envelope of $J$.

We did not check correctly this point in the previous version of the paper and the same minor mistake is made in Lemma 5.1 of \cite{lefevere_large_2011}. \hfill $\diamondsuit$
\end{rmq}

One can ask about whether the infimum defining $J_1$ is achieved or not, hence is a minimum. This question is briefly studied in Lemma 5.1 of \cite{lefevere_large_2011}, where the argument p.22, showing that $\pi_n$ therein is tight, sounds strange. Let us give a complete proof.

\begin{prop}
Under the assumptions of Proposition \ref{propJegal}, for $m\neq 0$, the infimum in \eqref{eqbarJ1} is a minimum, provided it is finite.
\end{prop}
\begin{proof} We use the expression \eqref{eqbarJbis} in order to prove the proposition. Assume that $m \neq 0$.
If $\bar J_1(m)<+\infty$ consider a minimizing sequence $(\gamma_n,\alpha_n,\nu'_n)_{n\ge0}$.
First, $H(\nu'_n|\psi)<+\infty$ (at least for large $n$'s), so that $\nu'_n$ is absolutely continuous with respect to $\psi$, and so $\nu'_n(|w|\leq K)=1$. It follows that $\gamma_n /\alpha_n \leq K/|m|$ hence $\gamma_n \leq K/|m|$. 

Since $\alpha_n\in]0,1]$ and $\gamma_n$ is bounded, one can find a subsequence still denoted $(\alpha_n,\gamma_n)_{n\ge0}$ converging to $(\alpha,\gamma)\in[0,1]\times [-K/|m|, K/|m|]$.  
In addition, for $n$ large enough, 
$$(\alpha_n/\gamma_n) H(\nu'_n|\psi) \leq \bar J_1(m) +1:=C$$ 
so that 
$$H(\nu'_n|\psi) \leq C \, (\gamma_n/\alpha_n) \leq C \, (K/|m|) \, .$$ 
Since the entropy is bounded, the sequence $(\nu'_n)_{n\ge0}$ is tight and one can thus also find a subsequence weakly converging to $\nu'_\infty$ which satisfies $H(\nu'_\infty|\psi) \leq \liminf_n H(\nu'_n|\psi) < +\infty$ thanks to the lower semi continuity of the entropy w.r.t. the first variable. 

Recall that $\gamma_n=\nu'_n(u)$. For all $M>0$, we have that $\gamma_n \geq \nu'_n(u\wedge M)$ and taking the limit in $n$, we deduce that $\nu'_\infty(u\wedge M)=\lim_n \, \nu'_n(u\wedge M) \leq \gamma$ and finally using the monotone convergence $\nu'_\infty(u)=\gamma' \leq \gamma$. 
We deduce in particular  that $\gamma >0$ since $\nu'_\infty(u=0)=0$ because the measure $\nu'_\infty$ is absolutely continuous w.r.t. $\psi$ and $\psi(u=0)=0$ by assumption.\\ 
Moreover, since $K \geq \gamma_n |m|/\alpha_n$ and $\gamma_n \to_{n\to\infty} \gamma >0$, we also have that  $\alpha=\lim_{n\to \infty} \alpha_n >0$. 
In addition, from the absolute continuity of $\nu'_n$ and $\nu'_\infty$ w.r.t. $\psi$, we deduce that $\nu'_\infty(|w|\leq K)=1$ and 
$$m\gamma/\alpha = \lim_n \nu'_n(w)= \lim_n \nu'_n(w \, \mathbf 1_{|w|\leq K})=\nu'_\infty(w \, \mathbf 1_{|w|\leq K}) = \nu'_\infty(w) \, .$$ Introduce $\nu_n = (1/\gamma_n) \, u \, \nu'_n$. $\nu_n$ is a sequence of probability measures that vaguely converges to $\nu_\infty$ satisfying $\nu_\infty(\mathcal X)=\gamma'/\gamma$, $\nu_\infty(1/u)=1/\gamma$ and $\nu_\infty(\varphi)=m/\alpha$. Of course $\bar \nu_\infty=\nu'_\infty$.

According to Lemma 2.3 and Lemma 2.2 in \cite{lefevere_large_2011} (based on the variational formula for the entropy) $$\liminf_n \, \frac 1\gamma_n \, H(\nu'_n|\psi) \geq (\gamma'/\gamma)  \frac{1}{\gamma} H(\nu'_\infty|\psi) + (1-(\gamma'/\gamma)) \theta_0 \, .$$ Finally define $\mu_\infty= \alpha \nu_\infty$ so that $\mu_\infty(\mathcal X) = \alpha (\gamma'/\gamma) \leq 1$ and $\mu_\infty \in {\bar{\mathcal M}}^1(\mathcal X)$. From what precedes we deduce 
\begin{eqnarray*}
\bar J_1(m) &=& \liminf_n \left((\alpha_n/\gamma_n) H(\nu'_n|\psi) + (1-\alpha_n)\theta_0\right)\\ &\geq& \mu_\infty(1/u) \, H(\bar \mu_\infty|\psi) + ((1-\alpha)+\alpha(1-(\gamma'/\gamma)) \theta_0 \\ &=& \mu_\infty(1/u) \, H(\bar \mu_\infty|\psi)+(1-\mu_\infty(\mathcal X))\theta_0
\end{eqnarray*}
and in addition $\mu_\infty(\varphi)=m$. Hence the infimum for $\bar J(m)$ is achieved at $\mu_\infty$. 
\end{proof}
\medskip

\section{Large deviations for the cumulative process when $W$ is bounded.}\label{secldcumul}

In this section, we shall deduce a LDP for $(Z_t/t)_{t\ge0}$ starting with \eqref{eqempircontrac}.  We still assume that $W$ is a bounded random variable, therefore it consists in relaxing the assumption on $\tau$ in corollary \ref{corcontrac}.

To this end, for $\varepsilon>0$, introduce $\tau^\varepsilon=\tau+\varepsilon$ and $\psi^\epsilon$ the distribution of $(\tau^\varepsilon,W)$. We then define $I^\varepsilon$ as in \eqref{eqI1}, replacing $\psi$ by $\psi^\varepsilon$, and ${\bar J}^\varepsilon$ as in \eqref{eqbarJ1} replacing $I$ by $I^\varepsilon$.

\begin{thm}\label{thmwbound}
Assume that there  exists a positive constant $K$ such that $\psi(|w|\leq K)=1$. Then, $(Z_t/t)_{t\ge0}$ satisfies a full LDP with the good convex rate function $\bar J$.
\end{thm}
\begin{proof}
The proof will be done in several steps.

\textit{Step1.} \quad We shall first prove the 
\begin{lemme}\label{lemwbound}
Assume that there  exists a positive constant $K$ such that $\psi-$almost surely, $|W|\leq K$. Then, $(Z_t/t)_{t\ge0}$ satisfies a weak LDP with the convex rate function 
\begin{equation}\label{eqtildeJ}
\widetilde J(m)= \sup_{\delta >0} \, \liminf_{\varepsilon \to 0} \, \inf_{|z-m|<\delta} \, {\bar J}^\varepsilon(z) \, .
\end{equation}
\end{lemme}
\begin{proof}[Proof of the lemma] \quad Following the same lines as in \eqref{eqempircontrac} 
$$\mu_t^{\varepsilon}(\varphi) = \frac 1t \, \sum_{i=1}^{M^\varepsilon_t} W_i \, + \, \frac{ (t-S^\varepsilon_{M^\varepsilon_t}) W_{M^\varepsilon_t+1}}{t \, \tau^\varepsilon_{M^\varepsilon_t+1}} \, .$$ 
Since $\tau^\varepsilon \geq \tau$, we deduce that $M_t^\varepsilon \leq M_t$. Accordingly
\begin{eqnarray*}
\left|\mu_t^
{\varepsilon}(\varphi) - \frac 1t \, \sum_{i=1}^{M_t} W_i \right| &\leq& \frac 1t \, \left|\sum_{i=M_t^\varepsilon +1}^{M_t} W_i\right| \, + \, \left|\frac{ (t-S^\varepsilon_{M^\varepsilon_t}) W_{M^\varepsilon_t+1}}{t \, \tau^\varepsilon_{M^\varepsilon_t+1}}\right| \\ &\leq& \frac Kt \, \left((M_t-M_t^\varepsilon) +1\right) \, .
\end{eqnarray*}
Using Theorem \ref{thmexpgood}, it is then sufficient to prove that $(M_t^\varepsilon/t)_\varepsilon$ is an exponentially good approximation of $M_t/t$ , i.e. that 
$$\lim_{\varepsilon \to 0} \, \limsup_{t \to \infty} \, \frac 1t \, \log \mathbb P(|M_t-M_t^\varepsilon| >\delta \, t) \, = \, -\infty \, .$$

The proof is similar to the one of \cite{lefevere_large_2011} Lemma 5.4 where a different approximation is used. Denote as usual by $\lfloor x \rfloor$ the integer part of $x\in\mathbb{R}$.
Recall that $M_t^\varepsilon \leq M_t$ and $S_n^\varepsilon=S_n+n\varepsilon$. Choose some $\delta >0$ and $A>0$. Then 
\begin{eqnarray*}
\mathbb P(M_t - \, M_t^\varepsilon > t\delta) &\leq& \sum_{n=1}^{\lfloor At\rfloor} \, \mathbb P(M_t - \, M_t^\varepsilon > t\delta \, ; \, M_t=n) \, + \, \mathbb P(M_t > \lfloor At\rfloor)\\ &=& \sum_{n=1}^{\lfloor At\rfloor} \, \mathbb P( M_t^\varepsilon < n - t\delta \, ; \, M_t=n) \, + \, \mathbb P(S_{\lfloor At\rfloor}\leq t) \\ &\leq& \sum_{n=1}^{\lfloor At\rfloor} \, \mathbb P( S^\varepsilon_{\lfloor n-t\delta\rfloor} \geq t ; \, M_t=n) \, + \, \mathbb P(S_{\lfloor At\rfloor}\leq t) \\ &\leq& \sum_{n=1}^{\lfloor At\rfloor} \, \mathbb P( S_{\lfloor n-t\delta\rfloor} \geq t -(n-t\delta)\varepsilon ; \, S_n=t) \, + \, \mathbb P(S_{\lfloor At\rfloor}\leq t)\\ &\leq& \sum_{n=1}^{\lfloor At\rfloor} \, \mathbb P( S_n- S_{\lfloor n-t\delta\rfloor} \leq(n-t\delta)\varepsilon) \, + \, \mathbb P(S_{\lfloor At\rfloor}\leq t)\\ &\leq& At \, \mathbb P(S_{\lfloor t\delta\rfloor}\leq At\varepsilon) + \mathbb P(S_{\lfloor At\rfloor}\leq t)\,,
\end{eqnarray*}
where we have used that the distribution of $S_j-S_k$ is the one of $S_{j-k}$ for any  positive integers $j\ge k$. 

According to Markov inequality 
$$\mathbb P(S_{\lfloor t\delta\rfloor}\leq At\varepsilon) =\mathbb P(e^{-S_{\lfloor t\delta\rfloor}/\varepsilon}\geq e^{-At})\leq \exp (At + \lfloor t\delta\rfloor \, \log \mathbb E(e^{-\tau/ \varepsilon}) )\, .$$ 
Thus
$$\limsup_{t \to \infty} \, \frac 1t \, \log(At \, \mathbb P(S_{\lfloor t\delta\rfloor}\leq At\varepsilon)) = A +  \delta \, \log \mathbb E(e^{-\tau/ \varepsilon}) \, .$$
Since $\log \mathbb E(e^{-\tau/ \varepsilon})\to_{\varepsilon\to0} -\infty$, 
we have 
$$\lim_{\varepsilon \to 0} \limsup_{t \to \infty} \, \frac 1t \, \log(At \, \mathbb P(S_{\lfloor t\delta\rfloor}\leq At\varepsilon)) = - \infty \, .$$
Similarly
$$\mathbb P(S_{\lfloor At\rfloor}\leq t) \leq \exp (t + \lfloor At\rfloor \, \log \mathbb E(e^{-\tau}))\,, $$ 
so that choosing $A$ large enough, we can make $\frac 1t \, \log \mathbb P(S_{\lfloor At\rfloor}\leq t)$ as small as we want i.e. less than $-B$ for any given $B>0$. It is then enough to let $\varepsilon$ go to $0$  and then $B$ go to infinity to obtain the result.
\end{proof}

In particular we know  from Theorem \ref{thmexpgood} that $\widetilde J$ is lower semi-continuous so that its level sets are closed.
\medskip

\textit{Step 2.} \quad We shall now identify $\widetilde J$ with $\bar{J}$. Recall that for all $m\neq 0$, $\bar{J}(m)=J(m) = \inf_{\beta>0} \sup_{x,y\in \mathbb R^2} \Lambda (m, \beta, x,y)$, where $\Lambda$ is defined in \eqref{eqcramer2}. 
\begin{lemme}\label{lemegal}
Under the assumptions of Lemma \ref{lemwbound}, $\widetilde J \geq \bar J$.
\end{lemme}
\begin{proof}[Proof of the Lemma]
Since $\tau > 0$ almost surely, one can find $x_\tau <0$ such that $\mathbb E(e^{x_\tau \, \tau})=e^{-1}$ so that 
$$\sup_{x,y} \Lambda(m,\beta,x,y) \geq \sup_x \Lambda(m,\beta,x,0) \geq x_\tau + \beta \, . $$ 
In particular if $J(m)<+\infty$ the infimum in $\beta$ has to be taken for $\beta \leq J(m) - x_\tau=\beta_\tau$. 
\medskip

From now we assume that $\widetilde J < +\infty$, indeed if $\widetilde J(m)=+\infty$, the inequality $\bar J(m)\leq \widetilde J(m)$ clearly holds.
The key remark is the following equality 
\begin{equation}\label{eqlambda}
\Lambda^\varepsilon(m,\beta,x,y)=\Lambda(m,\beta,x,y) -x\beta\varepsilon \, .
\end{equation}
If $\theta_0<+\infty$ it immediately follows from \eqref{eqlambda} and the fact that according to the proof of Proposition \ref{propJegal} the supremum in $\bar{J}$ can be restricted to $\{x\le \theta_0\}$ that $$\bar J(m) \leq  \bar J^\varepsilon(m) + \beta_\tau \, \varepsilon \, \theta_0\,,$$
 for the case $m=0$ just remark in addition that $\theta_0 \leq \theta_0(1+\beta_\tau \varepsilon)$. 

One can find a sequence $(m_n,\varepsilon_n)_{n\ge0}$ going to $(m,0)$ such that $\widetilde  J(m) = \liminf_{n \to \infty} \, \bar J^{\varepsilon_n}(m_n)$. Since $\bar J$ is lower semi continuous, 
 $$\bar J(m) \leq \liminf_{n \to \infty} \bar J(m_n) \leq \liminf_{n \to \infty} \, (\bar J^{\varepsilon_n}(m_n)+\theta_0 \beta_\tau \, \varepsilon_n)=\widetilde J(m) \, .$$

If $\theta_0=+\infty$ consider the previous sequence $(m_n,\varepsilon_n)_{n\ge0}$. One can in addition find a sequence $(\beta_n)_{n\ge0}$ and some sequence $(\eta_n)_{n\ge0}$ going to $0$ such that for all $(x,y)$, 
$$\Lambda(m_n,\beta_n,x,y) -x\beta_n\varepsilon_n \leq \widetilde J(m) + \eta_n \, .$$  
Since $\beta_n \in [0,\beta_\tau]$, we may assume that $\beta_n \to \beta$ up to considering a subsequence. $\beta$ has to be positive, otherwise, taking limits as $n\to\infty$ we  would get that for all $(x,y)$ 
$$\Lambda(m,0,x,y)=x+my \leq \widetilde J(m) < +\infty$$which is impossible. 
Hence $\beta>0$ and taking limits again, we obtain $\Lambda(m,\beta,x,y) \leq \widetilde J(m)$ for some $\beta>0$ and all $(x,y)$ so that $\bar J(m) \leq \widetilde J(m)$.
\end{proof}
\medskip

We turn to the converse inequality 
\begin{lemme}\label{lemegal2}
Under the assumptions of Lemma \ref{lemwbound}, $\widetilde J \leq \bar J$.
\end{lemme}
\begin{proof}
It is enough this time to assume that $\bar J(m)$ and thus $J(m)$ is finite. Notice furthermore than if $m=0$ and $\bar J(0)=\theta_0$ there is nothing to prove since $\widetilde J(0) \leq \liminf_{\varepsilon \to 0} \bar J^\varepsilon(0) \leq \theta_0$.
As a consequence if $m=0$ we may assume in addition that $J(0) < \theta_0$.

Recall that $\bar{J}$ is defined in \eqref{eqbarJ1}. Let $\mu_k$ be a minimizing sequence of $\bar J(m)$ in ${\bar{\mathcal M}}^1(\mathcal X)$, i.e. $I(\mu_k) \leq \bar J(m) + \eta_k$ with $\eta_k \to_{k\to\infty} 0$ and $\mu_k(\varphi)=m$. From the definition of $I$, we have in particular $\mu_k(1/u) < +\infty$. Let us introduce $\mu_k^\varepsilon$ the push forward of $\mu_k$ by the application $t_\varepsilon:(u,w) \mapsto (u+\varepsilon,w)$ (i.e. if $(\tau,W)$ is distributed according to $\mu_k$, $\mu_k^\varepsilon$ is the distribution of $(\tau + \varepsilon,W)$). 
Of course $\mu_k^\varepsilon(\mathcal X) \to_{\varepsilon\to0} \mu_k(\mathcal X)$ and $\mu_k^\varepsilon(1/u) \to_{\varepsilon\to0} \mu_k(1/u)$ thanks to Lebesgue's bounded convergence theorem, and finally, since $W$ is bounded for all considered measures, the same theorem shows that $$\mu_k^\varepsilon(\varphi)=m_k^\varepsilon \to m=\mu_k(\varphi) \quad \textrm{as $\varepsilon \to 0$.}$$ 
Since the minimizing measure is not the null measure, we may assume that $\mu_k(\mathcal X) \geq \chi >0$ for all $k$, so that $ H(\bar \mu_k|\psi)<+\infty$.

In addition, we have for any bounded continuous function $f$
\begin{eqnarray*}
\int \, f(u,w) \, \bar \mu_k^\varepsilon(du,dw) &=& \int \, f(u,w) \frac{1}{\mu^\varepsilon_k(1/u)} \, \frac 1u \, \mu^\varepsilon_k(du,dw) \\ &=& \int \, f(u+\varepsilon,w) \; \frac{1}{\mu_k(1/(u+\varepsilon))} \, \frac{1}{u+\varepsilon} \, \mu_k(du,dw) \\&=& \int \, f(u+\varepsilon,w) \; \frac{\mu_k(1/u)}{\mu_k(1/(u+\varepsilon))} \, \frac{u}{u+\varepsilon} \, \bar \mu_k(du,dw)
\end{eqnarray*}
Since $1/(u+\varepsilon) \leq 1/u$ which is $\mu_k$ integrable and $u/u+\varepsilon \leq 1$, it is thus immediately seen, thanks to Lebesgue's convergence theorem, that $\bar \mu_k^\varepsilon \to \bar \mu_k$ (and of course $\psi^\varepsilon \to \psi$) weakly as $\varepsilon \to 0$.  

Since $H(\bar \mu_k|\psi)<+\infty$, $\bar \mu_k$ is absolutely continuous w.r.t. $\psi$ with a density denoted by $g_k$. It follows that $\bar \mu_k^\varepsilon$ is absolutely continuous w.r.t. $\psi^\varepsilon$ with a density given by $$g_k^\varepsilon(u,w) = \frac{\mu_k(1/u)}{\mu_k(1/(u+\varepsilon))} \, \frac{u-\varepsilon}{u} \, g_k(u-\varepsilon,w) \, = \, C^\varepsilon \, \frac{u-\varepsilon}{u} \, g_k(u-\varepsilon,w) \, ,$$ recall that $\psi^\varepsilon(u>\varepsilon)=1$ so that we only need to consider such $u$'s.

We thus have
 $$H(\bar \mu_k^\varepsilon|\psi^\varepsilon) = \int \, g_k^\varepsilon \, \log g_k^\varepsilon \,  \, d\psi^\varepsilon = \int \, \log\left(C^\varepsilon \, \frac{u}{u+\varepsilon} \, g_k(u,w)\right) \;  C^\varepsilon \, \frac{u}{u+\varepsilon} \, g_k(u,w) \,\psi(du,dw) \, .$$ 
 Notice that, for $\varepsilon \leq 1$, $C^\varepsilon \, \frac{u}{u+\varepsilon} \, g_k(u,w) \leq 
C^1 \, g_k(u,w)$ so that 
$$\left|\log\left(C^\varepsilon \, \frac{u}{u+\varepsilon} \, g_k(u,w)\right) \;  C^\varepsilon \, \frac{u}{u+\varepsilon} \, g_k(u,w)\right| \, \leq \, \max\left(e^{-1} \, ; \, \log(C^1 \, g_k(u,w)) \, C^1 \, g_k(u,w)\right)$$ 
which is $\psi$ integrable since $H(\bar \mu_k|\psi)<+\infty$. It follows, using again Lebesgue's theorem, that $\lim_{\varepsilon \to 0} \, H(\bar \mu_k^\varepsilon|\psi^\varepsilon) = H(\bar \mu_k|\psi)$.

For a given $\delta>0$, we thus have $$\liminf_{\varepsilon \to 0} \, \inf_{|z-m|<\delta} \, \bar J^\varepsilon(z) \leq \liminf_{\varepsilon \to 0} \, \bar J^\varepsilon(m_k^\varepsilon) \leq J(m) + \eta_k \, .$$ The upper bound does not depend on $\delta$ and it remains to make $\eta_k \to 0$ to get the result.
\end{proof}
\medskip

\textit{Step 3.} \quad In oder to get the full LDP we need to check condition (3) in Theorem \ref{thmexpgood} i.e. that for all closed set $F$ , $$\inf_{z \in F} \bar J(z) \leq \limsup_{\varepsilon \to 0} \inf_{z \in F} \bar J^{\varepsilon}(z) \, .$$ We may of course assume that the right hand side is finite. For $\theta_0 < +\infty$ it is an immediate consequence of $\bar J(m) \leq \bar J^{\varepsilon}(m) + \beta_\tau \theta_0 \, \varepsilon$. 

If $\theta_0=+\infty$, remark that for $\beta <\beta_{\tau}$ $$\sup_{x,y} \Lambda(m,\beta,x,y) \geq \Lambda(m,\beta,0,1)=m-\beta \log \mathbb E(e^{W}) \geq m - \beta  K \, \geq m - \beta_\tau  K \, ,$$ and similarly  $$\sup_{x,y} \Lambda(m,\beta,x,y) \geq \Lambda(m,\beta,0,-1)=-m-\beta \log \mathbb E(e^{-W}) \geq - m - \beta_\tau  K \, .$$ It follows $J^\varepsilon(m) \geq |m| - \beta_\tau  K$ for all $\varepsilon$ (including $\varepsilon=0$), so that the level sets $\bar J^\varepsilon \leq M$ are all included in the ball $|m| \leq M + \beta_\tau K$.  

For a closed set F, one can thus find a sequence $\varepsilon_n,z_n$ with $\varepsilon \to_{n\to\infty} 0$ such that $\bar J^{\varepsilon_n}(z_n) \leq \inf_{z' \in F} J^{\varepsilon_n}(z') +1/n$ and $z_n \in F \cap \{|m|\leq C\}$ for some $C$ large enough. Taking a subsequence if necessary, we may assume that $z_n \to z \in F$ since $F$ is closed. We have $\bar J^{\varepsilon_n}(z) \geq \bar J^{\varepsilon_n}(z_n) - (1/n)$. We can thus argue as in the proof of Lemma \ref{lemegal} to show that $$\limsup_n \inf_{z' \in F} J^{\varepsilon_n}(z') = \limsup_n \bar J^{\varepsilon_n}(z) \geq \bar J(z) \geq \inf_{z' \in F} \bar J(z') \, .$$
\end{proof}
\medskip

\section{Deviations for the cumulative process in the general case. Proof of Theorem \ref{thmmain}.}\label{secgeneral}

We will now try to relax the boundedness assumption on $W$. We thus introduce $W^n=W\vee (-n) \wedge n$, $\psi^n$ the distribution of $(\tau,W^n)$, $I^n$, $\bar J^n$ and $J^n$ are defined accordingly. It is thus natural to look at 
\begin{equation}\label{eqtildegene}
\widetilde J(m)= \sup_{\delta>0} \, \liminf_{n \to +\infty} \, \inf_{|z-m|<\delta} \bar J^n(z) \, .
\end{equation}
We shall this time first compare $\widetilde J$ and $\bar J$.

\begin{lemme}\label{lemtilde1}
It holds $\bar J \leq \widetilde J$.
\end{lemme}

\begin{proof}
As in the proof of Lemma \ref{lemegal}, $\sup_{x,y}\Lambda^n(m,\beta,x,y) \geq x_\tau +\beta$ so that if $J^n(m)<+\infty$ the infimum in $\beta$ has to be taken for for $\beta \leq J^n(m) - x_\tau$. 

If $\widetilde J(m) < +\infty$ one can find a sequence $(m_n,\beta_n)_{n\ge0}$ such that $m_n \to m$, $\beta_n \in(0, \beta_\tau]$ where $\beta_\tau \leq \widetilde J(m) +1 -x_\tau$ and for $n$ large enough and all $(x,y)$, $$x + m_n y - \beta_n \log \mathbb E(e^{x \tau + y W^n}) \leq \widetilde J(m) + 1/n \, .$$ Taking a subsequence if necessary we may assume that $\beta_n \to \beta_\infty$. 

We want to pass to the limit in the previous inequality. We may assume that $\mathbb E(e^{x\tau}) < +\infty$, otherwise, for all $\beta>0$, 
$$x+my - \beta \log \mathbb E(e^{x \tau + y W})=-\infty \, .$$ 
Since $e^{x\tau+yW^n} \, \mathbf 1_{y W^n\leq 0}=e^{x\tau+yW^n} \, \mathbf 1_{y W\leq 0}$ is dominated by $e^{x\tau} \, 1_{y W\leq 0}$, which is assumed to be integrable, we may apply the bounded convergence theorem and get $\lim_n \mathbf E(e^{x\tau+yW^n} \, \mathbf 1_{y W^n\leq 0}) =\mathbb E(e^{x\tau+yW} \, \mathbf 1_{y W\leq 0})$. The other part, $\lim_n \mathbf E(e^{x\tau+yW^n} \, \mathbf 1_{y W^n> 0}) =\mathbb E(e^{x\tau+yW} \, \mathbf 1_{y W> 0})$ is a consequence of the monotone convergence theorem. 

We may thus conclude that for all $(x,y)$, 
$$ x+my - \beta_\infty \log \mathbb E(e^{x \tau + y W}) \leq \widetilde J(m) \, ,$$ 
hence $J(m) \leq \widetilde J(m)$, provided $\beta_\infty >0$. If $\beta_\infty=0$ we have obtained that for all $x$ such that $\mathbb E(e^{x\tau}) < +\infty$, $x+my \leq \widetilde J(m)$ which is impossible if $m\neq 0$, or if $m=0$ and $\widetilde J(0)<\theta_0$. Since $\bar J(0) \leq \theta_0$, the case $\widetilde J(0) \geq \theta_0$ is immediate.
\end{proof}
For the converse
\begin{lemme}\label{lemtilde2}
It holds $\bar J \geq \widetilde J$.
\end{lemme}
\begin{proof}
We shall follow the same route as for the proof of Lemma \ref{lemegal2}. We may similarly assume that $J(m)$ is finite and $J(0)<\theta_0$, so that the minimizing measure is not the null measure. We then consider a sequence $\mu_k$ such that $I(\mu_k)\leq \bar J(m) + \eta_k$, and we may assume again that $\mu_k(\mathcal X)\geq \chi>0$ for all $k$ so that $\sup_k H(\bar \mu_k|\psi) < +\infty$. 

We may decompose $\psi^n$ as $$\psi^n(du,dw)= \mathbf 1_{|w|<n} \psi(du,dw) + \gamma^n_+(du,dw) + \gamma^n_-(du,dw)$$ where $\gamma^n_+$ is the joint law of $(\tau, n \, \mathbf 1_{W\geq n})$ and $\gamma^n_-$ is the joint law of $(\tau, -n \, \mathbf 1_{W\leq -n})$. Of course $\psi^n$ weakly converges towards $\psi$.

We now introduce $\mu_k^n= \mathbf 1_{|w|<n} \, \mu_k$ so that $$\bar \mu^n_k= \frac{\mu_k(1/u)}{\mu_k(\mathbf 1_{|w|<n} \, (1/u))} \, \psi(|w|<n) \, \frac{d\bar \mu_k}{d\psi} \, \mathbf 1_{|w|<n} \, \psi^n \, .$$ It is then easily seen that $\mu^n_k$ weakly converges to $\mu_k$, that $\mu_k^n(\varphi)=m_k^n$ converges to $\mu_k(\varphi)=m$ and finally since $\mathbf 1_{|w|<n} \, \psi^n =\mathbf 1_{|w|<n} \, \psi$, denoting by $$c_k^n= \frac{\mu_k(1/u)}{\mu_k(\mathbf 1_{|w|<n} \, (1/u))} \, \psi(|w|<n)$$ that 
$$
H(\bar \mu_k^n|\psi^n)=\int c_k^n \, \log \left(c_k^n \, \frac{d\bar \mu_k}{d\psi}\right) \, \mathbf 1_{|w|<n} \, d\bar \mu_k $$ goes to $H(\bar \mu_k|\psi)$ as $n$ goes to infinity. We may thus conclude as in the proof of Lemma \ref{lemegal2}.
\end{proof}
\medskip

In order to get an LDP result for $(Z_t/t)_{t\ge0}$ it remains to study the approximation of $(Z_t/t)_{t\ge0}$ by $\{(Z^n_t/t)_{t\ge0}\}_{n\in\N}$. We may decompose
\begin{equation}\label{eqdiff1}
|Z_t - Z_t^n| = \sum_{i=1}^{M_t} (W_i-n)_+ \, + \, \sum_{i=1}^{M_t} (W_i+n)_-\,,
\end{equation}
 where $u_+=\max(u,0)$ and $u_-=\max(-u,0)$. We then have
\begin{lemme}\label{lemreduc1}
Assume that $\theta_0>0$ and $\eta_0>0$. For all $\delta >0$, $$\lim_{n \to \infty} \, \limsup_{t \to \infty} \, \frac 1t \, \log \mathbb P\left(\left|\frac{Z_t}{t} - \frac{Z^n_t}{t}\right|> {2}\delta\right) \leq  - \, \frac{\eta_0 \, \delta}{2} \, .$$ 
In particular if $\eta_0=+\infty$, $\{(Z^n_t/t)_{t\ge0}\}_{n\in\N} $ is an exponentially good approximation of $(Z_t/t)_{t\ge0}$.
\end{lemme}
\begin{proof}
Since $\eta_0$ and $\theta_0$ are positive, $\mathbb E(\tau)$ and $\mathbb E(|W|)$ are both finite.

From \eqref{eqdiff1}, we deduce that 
$$P\left(\left|\frac{Z_t}{t} - \frac{Z^n_t}{t}\right|> 2\delta\right) \le \mathbb P\left(\sum_{i=1}^{M_t} (W_i-n)_- > \delta t\right)+\mathbb P\left(\sum_{i=1}^{M_t} (W_i-n)_+ > \delta t\right)
$$
Note that using the elementary $\log(a+b) \leq \max(\log(2a),\log(2b))$ it is enough to look at $$\mathbb P\left(\sum_{i=1}^{M_t} (W_i-n)_+ > \delta t\right) \, ,$$ since the other term can be treated similarly.

Using that the $(W_i)_{i\ge1}$'s are i.i.d. we may write for $\delta>0$ and $c>0$, (as usual an empty sum is equal to $0$ by convention)
\begin{align*}
&\Proba \left( \sum_{i=1}^{\rnw_t} (W_i - n)_+ > \delta  t \right)\\
&\leq \Proba \left( \sum_{i=1}^{\lfloor ct\rfloor} (W_i - n)_+ > \frac{\delta t}{2} \right) + \Proba \left(\sum_{i=\lfloor ct\rfloor +1}^{\rnw_t} (W_i - n)_+ > \frac{\delta t }{2} \right)\\
&\leq  \Proba \left( \sum_{i=1}^{\lfloor ct\rfloor} (W_i - n)_+ > \frac{\delta t}{2} \right) + \Proba \left( \left\{\sum_{i=\lfloor ct\rfloor +1}^{\rnw_t} (W_i - n)_+ > \frac{\delta t }{2} \right\} \cap \left\{ 1+\lfloor ct\rfloor \leq \rnw_t < 2\lfloor ct\rfloor \right\} \right) \\
& \qquad + \Proba \left(\left\{\sum_{i=\lfloor ct\rfloor +1}^{\rnw_t} (W_i - n)_+ > \frac{\delta t }{2} \right\} \cap \left\{ \rnw_t \geq 2\lfloor ct\rfloor \right\} \right) \\
& \leq 2 \Proba \left( \sum_{j=1}^{\lfloor ct\rfloor} (W_j - n)_+ > \frac{\delta t }{2}  \right) + \Proba \left(\rnw_t   \geq 2\lfloor ct\rfloor \right)
\end{align*}
\medskip 

\textbf{Study of $\Proba \left(\rnw_t   \geq 2\lfloor ct\rfloor\right)$ \, .} \quad Start with the second term in the sum above. According to theorem 2.3 in \cite{tiefeng_large_1994}, we know that $\rnw_t/t$ satisfies a LDP with rate function $J_{\tau}$ given by
\begin{align*}
	J_{\tau}(u) = \left\lbrace
	\begin{array}{l}
	\sup_{\lambda\in\mathbb{R}} \{ \lambda - u \log \E(\e^{\lambda \tau}) \} \text{ if } u \geq 0 \,,\\
	\infty \quad \quad \quad \quad \quad \quad \quad \quad \quad \text{if } u < 0\,.
	\end{array}
	\right.
\end{align*}
Notice that $J_\tau(u)=u \, \Lambda^*(1/u,0)$ for $u>0$. In addition (see Lemma 2.6 in \cite{tiefeng_large_1994}) the supremum is achieved for $\lambda \leq 0$ if $u \in (1/\mathbb E(\tau) \, , \, +\infty)$ and $J_\tau$ is non-decreasing on this interval.

It follows that for $2c > 1/\mathbb E(\tau)$,
\begin{equation}\label{eqldM}
\limsup_{t \rightarrow  \infty} \frac{1}{t}  \log \Proba \left(\rnw_t  \geq 2 \lfloor ct\rfloor\right) \le - J_{\tau}(\lfloor ct\rfloor) \, .
\end{equation}
In order to get $\lim_{n\to\infty} \, \limsup_{t \rightarrow + \infty} \frac{1}{t} \log  \Proba \left(\rnw_t \geq 2\lfloor ct\rfloor \right) \le -\infty$ for some sequence $c_n$ (to be chosen later) it remains to show that 
	\begin{align*}
	J_{\tau}(u) \cvg{u}{\infty} + \infty.
	\end{align*}
Recall that  $x_\tau$ satisfies $\E(\e^{x_{\tau} \tau}) = \e^{-1}$, so that for $u\geq 0$, 
$$
	J_{\tau}(u) = \sup_{\lambda\in\mathbb{R}} \{ \lambda - u \log \E(\e^{\lambda \tau}) \}\geq x_{\tau} - u \log \E(\e^{x_{\tau} \tau})
	\geq  u + x_{\tau} 
$$
yielding the desired result.	
\medskip	 

\textbf{Study of $\Proba \left( \sum_{j=1}^{\lfloor ct \rfloor} (W_j - n)_+ > \frac{\delta t }{2}  \right)$.} We handle this term with Cramer's theorem. Defining
\begin{align*}
\Psi_n (\lambda) = \log \E \left[ \e^{\lambda(W-n)_+} \right]\,,\\
\Psi_n^* (x) = \sup_{\lambda\in\mathbb{R}} \lbrace \lambda x - \Psi_n(\lambda)  \rbrace\,,
\end{align*}
we have
\begin{eqnarray*}
 \limsup_{t \rightarrow \infty} \frac{1}{t} \log \Proba \left( \sum_{j=1}^{\lfloor ct \rfloor} (W_j - n)_+ >  \delta t/2\right) &=& 
 \limsup_{t \rightarrow \infty} \frac{c}{\lfloor ct \rfloor} \log \Proba \left( \sum_{j=1}^{\lfloor ct \rfloor} (W_j - n)_+ >  \delta t/2\right) \\ &\leq& \limsup_{t \rightarrow \infty} \frac{c}{\lfloor ct \rfloor} \log \Proba \left( \sum_{j=1}^{\lfloor ct \rfloor} (W_j - n)_+ >  \delta \lfloor ct \rfloor/2c\right) \\ &\leq& - \, c \, \inf_{x \in [\delta/2c, + \infty)} \Psi_n^*(x).
\end{eqnarray*}
As the function $x \mapsto \Psi_n^*(x)$ is non-decreasing on $[\mathbb E((W-n)_+),+\infty)$, we have 
\begin{align*}
\limsup_{t \rightarrow \infty}\frac{1}{t} \log \Proba \left( \sum_{j=1}^{\lfloor ct \rfloor} (W_j - n)_+ > \delta t/2 \right)  \le - \, c \, \Psi_n^*(\delta/2c) \, ,
\end{align*}
provided $\delta/2c \geq \mathbb E((W-n)_+)$. Notice that for $\lambda<\eta_0$, 
$$c \, \Psi_n^*(\delta/2c) \, \geq \, \frac{\lambda \delta}{2} \, - \, 
c \, \log \left(1 + \mathbb E\left[ (e^{\lambda (W-n)} - 1) \, \1_{W>n}\right]\right)\,,
$$ 
Since both $\mathbb E((W-n)_+)$ and $\log \left(1 + \mathbb E\left[(e^{\lambda (W-n)} - 1) \1_{W>n}\right]\right)$ are going to $0$ as $n\to \infty$, it is always possible to choose $c_n$ growing to infinity such that as $n\to \infty$
$$c_n \, \mathbb E((W-n)_+) \to 0  \; \textrm{ and } \; c_n \, \log \left(1 + \mathbb E\left[(e^{\lambda (W-n)} - 1) \1_{W>n}\right]\right) \to 0 \,,$$
 We get 
$$\lim_{n\to\infty} \; \limsup_{t \rightarrow \infty} \frac{1}{t} \log \Proba \left( \sum_{j=1}^{\lfloor ct \rfloor} (W_j - n)_+ >  \delta t/2\right) \le - \,  \frac{\lambda \delta}{2} \, .$$ 
We may optimize in $\lambda$ and plug the same sequence $c_n$ in \eqref{eqldM} completing the proof. 
\end{proof}
\medskip

We will use the previous lemma to deduce
\begin{cor}\label{cortight}
Under the assumptions of Lemma \ref{lemreduc1}, $(Z_t/t)_{t \geq 0}$ is exponentially tight, i.e. for all $\alpha>0$, there exists a compact set $K_{\alpha}$ such that
	$$\limsup_{t \to \infty} \frac{1}{t} \log \Proba \left( \frac{Z_t}{t} \notin K_{\alpha}^c  \right) < - \alpha.$$
\end{cor}
\begin{proof}
	Since $Z^n_t/t$ is an approximation of $Z_t/t$ and satisfies a full LDP according to Theorem \ref{thmwbound}, we can decompose the probability as following: for each $n$, and for all $\delta$: 
	\begin{align}
	\Proba \left( \frac{Z_t}{t} \notin [-A, A] \right) 
	&\leq \Proba \left(\left|\frac{Z_t}{t}- \frac{Z^n_t}{t} \right| > \delta \right) + \Proba \left( \frac{Z^n_t}{t} \notin [-A+ \delta, A- \delta] \right)  \nonumber \\
	&\leq \Proba \left( \left|\frac{Z_t}{t}- \frac{Z^n_t}{t} \right| > \delta \right) + \Proba \left( \frac{Z^n_t}{t} < - A + \delta \right)  + \Proba \left( \frac{Z^n_t}{t} > A - \delta \right).\nonumber\\
	&\leq 3 \max\left(\Proba \left( \left|\frac{Z_t}{t}- \frac{Z^n_t}{t} \right| > \delta \right) ,\Proba \left( \frac{Z^n_t}{t} < - A + \delta \right) ,\Proba \left( \frac{Z^n_t}{t} > A - \delta \right)\right).\label{eq:maj1_exptight}
	\end{align}
	
	By Lemma \ref{lemreduc1}, $Z^n_t/t$ and $Z_t/t$ satisfies
	$$	\forall \delta>0, \lim_{n\to\infty}\limsup_{t\to\infty} \frac{1}{t} \log \Proba \left( \left|\frac{Z_t}{t}- \frac{Z^n_t}{t} \right| > \delta \right) = - \frac{\eta_0 \delta}{4},$$
	i.e. 
	\begin{align}\label{eq_nalphadelta}
	\forall \alpha >0, \forall \delta>\frac{2 \alpha}{\eta_0}, \exists n(\alpha, \delta), \forall n> n(\alpha, \delta), \limsup_{t\to\infty}  \frac{1}{t} \log \Proba \left( \left|\frac{Z_t}{t}- \frac{Z^n_t}{t} \right| > \delta \right) \leq - \alpha.
	\end{align}

	We just have to study $\Proba \left( \frac{Z^n_t}{t} > A - \delta \right)$ and the symmetric case. We know from Theorem \ref{thmwbound} that: 
	\begin{align*}
	\limsup_{t\to\infty} \frac{1}{t} \log \Proba \left( \frac{Z^n_t}{t} > B \right) \leq - \inf_{m > B} \bar J^n(m).
	\end{align*}
Since $\bar J^n$ has compact level sets, for all $\alpha >0$ one can choose a level $B_{\alpha}$ such that $\forall m > B_{\alpha}, J^n(m)> \alpha$. The result follows by choosing $A=B_\alpha +\delta$.
\end{proof}
\medskip
	
\paragraph{Proof of Theorem \ref{thmmain}}
In the case where $\eta_0=+\infty$, using the approximation $W^n$, Lemma \ref{lemtilde1} and Lemma \ref{lemtilde2} allow to obtain the weak LDP. The full LDP derives from Corollary \ref{cortight} combined with Lemma \ref{lem_exptightLDP}.
\medskip

If $\eta_0 <+\infty$ we only obtain asymptotic deviation bounds. Recall that $m=\mathbb E(W)/\mathbb E(\tau)$ is the limit of $Z_t/t$ as $t \to +\infty$.For all $\kappa \in (0,1)$ and $a>0$, it holds $$\mathbb P\left(\frac{Z_t}{t} \geq m+a\right) \leq \mathbb P\left(\frac{Z^n_t}{t} \geq m+\kappa a\right) +\, \mathbb P\left(\left|\frac{Z_t}{t} - \frac{Z^n_t}{t}\right| \geq (1-\kappa)a\right)\,,$$ 
so that, for all $n\ge0$,
\begin{align*}
\limsup_{t \to \infty} &\log \mathbb P\left(\frac{Z_t}{t} \geq m+a\right) \\
&\leq \; \max \left[ \limsup_{t \to +\infty} \log \mathbb P\left(\frac{Z^n_t}{t} \geq m+\kappa a\right) \, ; \, \limsup_{t \to \infty} \log \mathbb P\left(\left|\frac{Z_t}{t} - \frac{Z^n_t}{t}\right| \geq (1-\kappa)a\right)  \right] \, .
\end{align*} 
Taking the $\liminf$ in $n$ we deduce 
\begin{eqnarray*}
\limsup_{t \to \infty} \log \mathbb P\left(\frac{Z_t}{t} \geq m+a\right) &\leq& \max \left[\liminf_{n\to\infty} (- \inf_{z \geq m+\kappa a} \bar J^n(z)) \, ; \, - \, \frac{\eta_0 \, (1-\kappa a)}{4} \right]\\ &\leq& - \, \min \left[\limsup_{n\to\infty} (\inf_{z \geq m+\kappa a} \bar J^n(z)) \, ;   \, \frac{\eta_0 \, (1-\kappa a)}{4} \right] \, .
\end{eqnarray*}
To complete the proof of the Theorem it is enough to prove
\begin{lemme}\label{lemdefin}
Assume $\eta_0>\infty$, then for any $z_0\in\mathbb R$, 
$$\limsup_{n\to\infty} (\inf_{z \geq z_0} \bar J^n(z)) \geq \inf_{z \geq z_0} \bar J(z) \, .$$
\end{lemme}
\begin{proof}
The proof is close to the one of Lemma \ref{lemtilde1}. We may of course assume that the left hand side is finite, denoted by $C(z_0)$. As usual, for a fixed $\varepsilon >0$, we may find a sequence $(z_n)_{n\ge0}$ such that for any $n\in\mathbb N$, $z_n\geq z_0$ and $\inf_{z \geq z_0} \bar J^n(z)+\varepsilon \geq \bar J^n(z_n)$, so that $\limsup_{n\to\infty} \bar J^n(z_n) \leq C(z_0)+\varepsilon$.

We want to show that the sequence $(z_n)_{n\ge0}$ is bounded. The key point is to remark that, taking the sign of $y$ into account
\begin{eqnarray*}
x+zy-\beta \log \mathbb E\left(e^{x\tau +y W^n}\right) &\geq& x+zy-\beta \log \mathbb E\left(e^{x\tau +|y| |W^n|}\right) \\ &\geq& x+zy-\beta \log \mathbb E\left(e^{x\tau +|y| |W|}\right)
\end{eqnarray*}
so that for all $n$, $$\bar J^n(z) \geq J^{|.|}(z):= \inf_{\beta>0} \, \sup_{x \in \mathbb R,y\geq 0} \, \left\{x+|z|y - \beta  \log \mathbb E\left(e^{x\tau +y |W|}\right)\right\} \, .$$ As before, taking $y=0$ we see that the infimum in $\beta$ has to be taken for $\beta \leq \beta_\tau=C(z_0)+1-x_\tau$, at least for $n$ large enough.

Taking $x=0$ we see that $J^{|.|}(z) \leq C(z_0)+\varepsilon$ implies $$|z| (\eta_0/2) \leq C(z_0)+ \beta_\tau \, \log \mathbb E\left(e^{(\eta_0/2) |W|}\right) \, ,$$  i.e $|z|\leq A$ for some positive $A$ that does not depend on $n$. This shows that $(z_n)_{n\ge0}$ is bounded, so that taking a subsequence if necessary $z_n \to z_{lim} \geq z_0$.

Consider $\bar J(z_{lim})$.  We may now mimic the proof of Lemma \ref{lemtilde1} replacing $m_n$ by $z_n$ and $m$ by $z_{lim}$, so that 
$$\inf_{z\geq z_0} \bar J(z) \leq \bar J(z_{lim}) \leq C(z_0)+ \varepsilon\,.$$
It remains to let $\varepsilon$ go to $0$.
\end{proof}
\medskip

\section{Application to Hawkes processes. Corrigendum.}

In \cite{cattiaux_costa_colombani} Theorem 2.12 and Corollary 2.13, we gave an application to Hawkes processes of our main results, with a \emph{wrong} bound. \\
As we have seen the correct one in Theorem 2.12 is $(1-\kappa)\theta_0 a/4$ ($\theta_0$ there is $\eta_0$ in the present paper), the factor $1/4$ is missing in \cite{cattiaux_costa_colombani}. The correct term in Corollary 2.13 is also $(1-\kappa)\theta_0 a/4$. Indeed according to equation (2.9) therein, $N_t^h=\hat N^h_t + R_t^h$ with $0\leq R_t^h \leq W_{M_t^h+1}$. If $W$ is bounded we may thus write $N_t^h=\mu^\varepsilon_t(\varphi) + A_t^\varepsilon$ where $A_t^\varepsilon \leq \frac Kt \, ((M_t-M_t^\varepsilon)+2)$, so that the proof of Theorem \ref{thmwbound} remains valid replacing $Z_t/t$ by $N^h_t/t$.

Also remark that we have to replace $J$ by $\bar J$, i.e. take care of the case $z=0$, even if here $m>0$ since $W\geq 0$ and $W\neq 0$.


\bigskip





\bigskip

\begin{thebibliography}{10}

\bibitem{asmussen}
S\o~ren Asmussen.
\newblock {\em Applied probability and queues}, volume~51 of {\em Applications
  of Mathematics (New York)}.
\newblock Springer-Verlag, New York, second edition, 2003.
\newblock Stochastic Modelling and Applied Probability.

\bibitem{borovkov_large_2015}
Alexander~A. Borovkov and Anatolii~A. Mogulskii.
\newblock Large deviation principles for the finite-dimensional distributions
  of compound renewal processes.
\newblock {\em Sib. Math. J.}, 56(1):28--53, 2015.

\bibitem{borovkov_large_2016_1}
Alexander~A. Borovkov and Anatolii~A. Mogulskii.
\newblock Large {Deviation} {Principles} for {Trajectories} of {Compound}
  {Renewal} {Processes}. {I}.
\newblock {\em Theory Probab. Appl.}, 60(2):207--224, 2016.

\bibitem{borovkov_large_2016_2}
Alexander~A. Borovkov and Anatolii~A. Mogulskii.
\newblock Large {Deviation} {Principles} for {Trajectories} of {Compound}
  {Renewal} {Processes}. {II}.
\newblock {\em Theory Probab. Appl.}, 60(3):349--366, 2016.

\bibitem{brown_asymptotic_1972}
Mark Brown and Sheldon~M. Ross.
\newblock Asymptotic {Properties} of {Cumulative} {Processes}.
\newblock {\em SIAM J. Appl. Math.}, 22(1):93--105, 1972.

\bibitem{cattiaux_costa_colombani}
Patrick Cattiaux, Laetitia Colombani, and Manon Costa.
\newblock Limit theorems for {H}awkes processes including inhibition.
\newblock {\em Stochastic Process. Appl.}, 149:404--426, 2022.

\bibitem{costa}
Manon Costa, Carl Graham, Laurence Marsalle, and Viet~Chi Tran.
\newblock Renewal in {H}awkes processes with self-excitation and inhibition.
\newblock {\em Adv. in Appl. Probab.}, 52(3):879--915, 2020.

\bibitem{Cs}
Imre Csisz\'{a}r.
\newblock Sanov property, generalized {$I$}-projection and a conditional limit
  theorem.
\newblock {\em Ann. Probab.}, 12(3):768--793, 1984.

\bibitem{dembo_large_2010}
Amir Dembo and Ofer Zeitouni.
\newblock {\em Large {Deviations} {Techniques} and {Applications}}.
\newblock Stochastic {Modelling} and {Applied} {Probability}. Springer-Verlag,
  Berlin Heidelberg, 2 edition, 2010.

\bibitem{duffy_how_2005}
Ken Duffy and Anthony~P. Metcalfe.
\newblock How to estimate the rate function of a cumulative process.
\newblock {\em J. Appl. Probab.}, 42(4):1044--1052, 2005.

\bibitem{glynn_limit_1993}
Peter~W. Glynn and Ward Whitt.
\newblock Limit theorems for cumulative processes.
\newblock {\em Stochastic Process. Appl.}, 47(2):299--314, 1993.

\bibitem{hawkes71}
Alan~G. Hawkes.
\newblock Spectra of some self-exciting and mutually exciting point processes.
\newblock {\em Biometrika}, 58:83--90, 1971.

\bibitem{kechris}
Alexander Kechris.
\newblock Classical Descriptive Set Theory, volume~156 of {\em Graduate Texts in Mathematics}
\newblock Springer-Verlag, New York, 1st edition, 1995

\bibitem{lefevere_large_2011}
Raphaël Lefevere, Mauro Mariani, and Lorenzo Zambotti.
\newblock Large deviations for renewal processes.
\newblock {\em Stochastic Process. Appl.}, 121(10):2243--2271, 2011.

\bibitem{mzsanov}
Mauro Mariani, and Lorenzo Zambotti.
\newblock A renewal version of {S}anov theorem..
\newblock {\em Electron. Commun. Probab.}, 19(69), 2014.

\bibitem{smith_1955}
Walter~L. Smith.
\newblock Regenerative stochastic processes.
\newblock {\em Proc. Roy. Soc. London Ser. A.}, 232(1188):6--31, 1955.

\bibitem{tiefeng_large_1994}
Jiang Tiefeng.
\newblock Large deviations for renewal processes.
\newblock {\em Stochastic Process. Appl.}, 50(1):57--71, 1994.

\bibitem{zamparo}
Marco Zamparo.
\newblock Large deviation principles for renewal-reward processes.
\newblock \newblock {\em Stochastic Process. Appl.}, 156:226--245, 2023.

\bibitem{zamparo2}
Marco Zamparo.
\newblock Large deviations in discrete time renewal theory.
\newblock \newblock {\em Stochastic Process. Appl.}, 139:80--109, 2021.

\end{thebibliography}
\end{document}